\documentclass[11pt]{amsart}
\usepackage{amsmath}
\usepackage{amssymb}
\usepackage[english]{babel}
\usepackage[latin1]{inputenc}
\usepackage[T1]{fontenc}
\usepackage{graphicx,epic,eepic}
\usepackage{color}
\usepackage{pdfpages}
\usepackage{enumerate}
\usepackage{amsthm}
\usepackage[all]{xy}
\usepackage{lastpage}
\usepackage{bbm}
\usepackage{url}
\usepackage{caption}
\usepackage{subcaption}
\newtheorem{Theo}{Theorem}[section]
\newtheorem{Def}[Theo]{Definition}
\newtheorem{Ex}[Theo]{Example}
\newtheorem{Lem}[Theo]{Lemma}
\newtheorem{Prop}[Theo]{Proposition}
\newtheorem{Cor}[Theo]{Corollary}

\newtheorem{Rem}[Theo]{Remark}
\newtheorem{Not}[Theo]{Notation}
\usepackage{tikz}
\usepackage{tikz-qtree}

\newcommand{\xymat}[1]{\begin{align*}\xymatrix{ #1}\end{align*}}

\title[Unstable $\nu_1$-periodic homotopy of $H$-spaces]{Unstable $\nu_1$-periodic homotopy of simply connected, finite $H$-spaces, using Goodwillie calculus}
\author{Jens Jakob Kjaer}
\date{\today}
\begin{document}

\thanks{The author thanks Mark Behrens for much help, as well as many personal conversations, especially the construction leading to Definition \ref{DefTheta} was formalized with a lot of help from this source. The author was partially supported by NSF DMS 1209387}

\begin{abstract}
In this paper we recover Bousfield's computation of $\nu_1$-periodic homotopy groups of simply connected, finite $H$-spaces from \cite{Bou99} using the techniques of Goodwillie calculus. This is done through first computing Andr\'{e}-Quillen cohomology over the monad  $\mathbb{T}$ that encodes the power operations of complex $K$-theory. Then lifting this computation to computing $K$-theory of topological Andr\'{e}-Quillen cohomology, and then using results of Behrens and Rezk relating it back to the Bousfield-Kuhn functor.

The fact that we recovers the result of Bousfield allows us to conclude $\nu_1$-periodic Goodwillie tower for simply connected, finite $H$-spaces converges.
\end{abstract}

\maketitle
\tableofcontents

\section{Introduction}
Famous results of Quillen \cite{Qui69} and Sullivan  \cite{Sul77} tell us that we can model the homotopy theory of $1$-connected rational spaces as either differential graded Lie algebras, or differential graded commutative algebras. We can translate between the two models when $X$ is a $1$-connected finite space, by
\begin{align*}
\pi_*(X)\otimes \mathbb{Q}\simeq AQ^*(\Lambda_X)
\end{align*}
where $AQ^*$ is Andr\'{e}-Quillen cohomology, and $\Lambda_X$ is a strictly commuting model for $C^*(X;\mathbb{Q})$. A natural question is to the extend this to more general spaces. It is a result of Mandell \cite{Man01} that the analogous statement fails for $p$-completed spaces.

In stable homotopy theory rational localization fits into chromatic homotopy theory as the height zero case. The unstable picture is a bit more complicated, but we do have the notion of chromatic localization of the category of spaces. In work to appear, Heuts \cite{Heu18} showed that $\nu_h$-periodic spaces are equivalent to $\nu_h$-periodic Lie algebras of spectra. In the same article it is shown that $\nu_h$-periodic commutative algebras of spectra fail to model $\nu_h$-periodic spaces. The failure comes down to the failure of the $\nu_h$-periodic Goodwillie spectral sequence to converge.

If the $\nu_h$-periodic Goodwillie spectral sequence converges for a space $X$, we follow \cite{BeRe17Survey}, and say that the space is $\Phi_h$-good. Here $\Phi_h$ is the Bousfield-Kuhn functor at height $h$ (see \cite{Kuh08} for an overview). In \cite{Kuh07} Kuhn interpreted the computation of Arone and Mahowald from  \cite{AM99} to say that the spheres are $\Phi_h$-good for all $h$.

Behrens and Rezk \cite{BeRe17} showed that the convergence of the Goodwillie tower is equivalent statement of
\begin{align*}
\Phi_h(X) \to TAQ(S_{K(h)}^X)
\end{align*}
being an equivalence, where $K(h)$ is the height $h$ Morava $K$-theory spectrum, and $S^X$ is the Spanier-Whitehead dual of $X$ as a ring spectrum. In \cite{BeRe17Survey} the authors further gave a class of spaces that are $\Phi_h$-good for all $h$.

In this current paper the author will focus on height $h=1$, at an odd prime $p$. The $\nu_1$-periodic homotopy groups of many spaces were computed by a number of authors using various methods (see \cite{Dav95} for an overview). Here we will focus on recovering the result of Bousfield in \cite{Bou99} concerning the computation of $\nu_1^{-1}\pi_*$ for $1$-connected $H$-spaces, using Goodwillie calculus, and as a result we will conclude that all $1$-connected $H$-spaces are $\Phi_1$-good. Our computation of differentials in the $\nu_1$-periodic Goodwillie spectral sequence enhances our understanding of the unlocalized Goodwillie spectral sequence of these spaces.

The structure of the paper is as follows: Section 2 contains the definitions of Andr\'{e}-Quillen cohomology both over algebras and monads, and some filtrations giving rise to spectral sequences. Section 3 recalls the definition of the monad $\mathbb{T}$, and its connection to the commutative operad. At the end of section 3, we will specialize to the height 1 case. Section 4 proves certain differentials in the $\mathbb{T}$ spectral sequence. Section 5 lifts the differentials from the $\mathbb{T}$ spectral sequence to the $TAQ$ spectral sequence. Finally, section 6 puts these pieces together and recovers a theorem of Bousfield, and concludes that the class of spaces it pertains to are all $\Phi_1$-good.

\section{Background material}
Throughout the paper, all spectra are completed at an odd prime $p$.
We will need a good symmetric monoidal category of spectra, so take spectra to mean the category of symmetric spectra as developed in \cite{HSS00}. When $E$ is a commutative ring spectrum, then let $E-mod$ denotes the category of $E$-modules. 
\subsection{Operad and monad cohomology, and some spectral sequences}
For the purpose of this paper an operad is symmetric and reduced. So for an operad $\mathcal{O}$ in a symmetric monoidal category $\mathcal{C}$, we have $\mathcal{O}(n)$ is a $\Sigma_n$-object, $\mathcal{O}(0)=*$, and $\mathcal{O}(1)$ is the monoidal unit.
\begin{Def}
If $\mathcal{O}$ is an operad in $E-mod$ for a commutative ring spectrum $E$, then there is a functor 
$\mathcal{F}_\mathcal{O}: E-mod\to E-mod$ 
given by 
\begin{align*}
\mathcal{F}_\mathcal{O}(A)= \bigvee_i(\mathcal{O}(i)\wedge_E A^{\wedge_E i})_{h\Sigma_i}.
\end{align*}
We call $\mathcal{F}_\mathcal{O}$ the monad associated to $\mathcal{O}$, and $\mathcal{F}_\mathcal{O}(A)$ the free $\mathcal{O}$-algebra on $A$. 
\end{Def}
Note the fact that $\mathcal{O}$ is assumed reduced means that $\mathcal{F}_\mathcal{O}$ is augmented as a monad.

\begin{Def}
If $\mathcal{O}$ is an operad in $E-mod$ for a commutative ring spectrum $E$, $A$ is an  $\mathcal{O}$-algebra, let $(\ )^\vee=F(\ ,E)$ be the Spanier-Whitehead dual, and $\mathcal{F}_\mathcal{O}$ is the monad associated to $\mathcal{O}$, then 
\begin{align*}
TAQ^*_\mathcal{O}(A):=\pi_*B(1,\mathcal{F}_\mathcal{O},A)^\vee
\end{align*}
is the $\mathcal{O}$ cohomology of $A$, or the topological Andr\'{e}-Quillen cohomology of $A$. 
\end{Def}
For the original definitions and a more in depth discussion see \cite{Bas99}, \cite{BaMa11}, and \cite{Har10}.

In the category of spectra we can use the grading of $\mathcal{F}_\mathcal{O}$ by arity for an operad $\mathcal{O}$ to compute topological Andr\'{e}-Quillen cohomology of an $\mathcal{O}$-algebra.
\begin{Def}
We define inductively
\begin{align*}
\mathcal{F}_\mathcal{O}[k](A)&=  \big(\mathcal{O}(k)\wedge A^{\wedge k} \big)_{h\Sigma_k}\\
\mathcal{F}_\mathcal{O}^{\circ q}[k](A)&= \left(\bigvee_{k= n_1+\ldots +n_l}\mathcal{O}(l) \wedge \left( \mathcal{F}_\mathcal{O}^{\circ (q-1)}[ n_1](X) \wedge \ldots \wedge  \mathcal{F}_\mathcal{O}^{\circ (q-1)}[ n_l](X)\right) \right)_{h\Sigma_{k}}
\end{align*}
\end{Def}
This induces a filtration of the monadic bar construction, which we denote 
\begin{align*}
B(1,\mathcal{F}_\mathcal{O},A)[\leq k].
\end{align*}
Let $B(1,\mathcal{F}_\mathcal{O},A)[= k]$ denote the associated quotients.
\begin{Rem}
If $\mathcal{O}$ is an operad in spectra, and $A$ is an $\mathcal{O}$ algebra, then there is a spectral sequence:
\begin{align*}
\bigoplus_k \pi_* B(1,\mathcal{F}_\mathcal{O},A)[=k]^\vee \Rightarrow TAQ^*_\mathcal{O}(A)
\end{align*}
Natural in both $A$, and $\mathcal{O}$. We will call this spectral sequence the $TAQ_\mathcal{O}$ spectral sequence.
\end{Rem}

\begin{Def}
If $\mathcal{O}$ is an operad of spectra then define 
\begin{align*}
\mathcal{F}_\mathcal{O}^{dp}(X)=\bigvee (\mathcal{O}(i)\wedge X^{\wedge i})^{h\Sigma_i}.
\end{align*}
\end{Def}
The $dp$ stands for divided powers. If we replaced homotopy fixed points with actual fixed points we get a monad (alternatively one could work in an infinity category, and get a homotopy coherent monad), which in the algebraic setting for the commutative operad encodes divided powers \cite{Fre00}.

Recall that if $\mathcal{O}$ is a reduced operad, then it is an augmented monoid in the category of symmetric sequences with respect to the composition product, see for example \cite{MSS02}. Thus we can form the operadic bar construction, which is again a symmetric sequence $B(\mathcal{O})$. $B(\mathcal{O})(n)$ can be constructed as the space of rooted trees with $n$ leaves, with lengths of edges, and labels on the internal vertices coming from $\mathcal{O}$, see \cite{Chi05}.

\begin{Def}
If $A\in \mathcal{C}$ and $\mathbb{M}:\mathcal{C}\to \mathcal{C}$  an augmented monad, with augmentation $\epsilon: \mathbb{M}\to 1$. Then define $\overline{A}$ is the $\mathbb{M}$-algebra with the trivial action, i.e., the algebra on $A$ with structure map $\mathbb{M}(A)\stackrel{\epsilon(A)}{\to}1(A)=A$. 
\end{Def}

 Note that $B(1,\mathcal{F}_\mathcal{O},\overline{A})\simeq \mathcal{F}_{B(\mathcal{O})}(A)$. In \cite{Chi05} Ching showed that for an operad in spaces or spectra, the bar construction $B(\mathcal{O})$ is an cooperad, and its Spanier-Whitehead dual is therefore an operad. This operad is called the Koszul dual of $\mathcal{O}$.
\begin{Def}
If $\mathcal{O}$ is an operad of spectra, we define $\mathcal{KO}$ to be its Koszul dual operad in the sense of \cite{Chi05}.
\end{Def}

From  \cite{Fre04} we know that a similar story can be told for operads of chain complexes over a ring. By abuse of notation we give the following definition.
\begin{Def}
If $\mathcal{O}$ is an operad of chain complexes, we define $\mathcal{KO}$ to be its Koszul dual operad in the sense of \cite{Fre04}.
\end{Def}

\begin{Cor}
The $TAQ_\mathcal{O}$-spectral sequence has $E_1$-page $\pi_*\mathcal{F}_{\mathcal{KO}}^{dp}(A^\vee)$, if $A$ is a $\mathcal{O}$-algebra, which is finite as a spectrum.
\end{Cor}

\begin{Def}
If $R$ is a ring, $\mathbb{M}:R-Mod \to R-Mod$ is an augmented monad, and $A$ is a $\mathbb{M}$-algebra then 
\begin{align*}
AQ_\mathbb{M}^*(A):=H^*\mathrm{Hom}_{R-Mod}(B(1,\mathbb{M},A),R)
\end{align*}
is the $\mathbb{M}$-cohomology of $A$, or the Andr\'{e}-Quillen cohomology of $A$.

Similarly we can define the associated homology theory 
\begin{align*}
AQ^\mathbb{M}_*(A):=H_*B(1,\mathbb{M},A).
\end{align*}
\end{Def}

Let $\mathcal{C}$ be a symmetric monoidal category with monoidal product
$\otimes$ and unit $\mathbbm{1}$, and suppose also that $\mathcal{C}$ admits finite coproducts (denoted $\oplus$, with
initial object $0$), and that $\otimes$ distributes over coproducts. For convenience, we also assume that inclusions of direct summands are always monomorphisms in $\mathcal{C}$. From \cite{Rez12} we have the following definition.
\begin{Def}
 By an exponential monad, we mean a monad $\mathbb{M}: \mathcal{C}\to \mathcal{C}$ equipped with natural isomorphisms
\begin{align*}
\nu: \mathbbm{1} \to \mathbb{M}(0),\ \ \  & \zeta: \mathbb{M}(X)\otimes \mathbb{M}(Y) \to \mathbb{M}(X\oplus Y) 
\end{align*}
where $\zeta$ is a natural transformation of functors $\mathcal{C}\times \mathcal{C}\to \mathcal{C}$, with the property that $(\nu,\zeta)$ makes $\mathbb{M}: \mathcal{C}^\oplus\to \mathcal{C}^\otimes $ into a strong symmetric monoidal functor. Furthermore, we
require that every $\mathbb{M}$-algebra, $A$, is naturally a commutative monoid object, with unit
\begin{align*}
\mathbbm{1}\stackrel{\nu}{\to} \mathbb{M}(0) \stackrel{\mathbb{M}(0)}{\to} \mathbb{M}(A) \to A ,
\end{align*}
and multiplication
\begin{align*}
A\otimes A \stackrel{}{\to} \mathbb{M}(A)\otimes \mathbb{M}(A)\stackrel{\zeta}{\to} \mathbb{M}(A\oplus A) \stackrel{\mathbb{M}(\nabla)}{\to} \mathbb{M}(A) \to A 
\end{align*}
\end{Def}
The canonical example of an exponential monad is the free commutative algebra monad
on the category of abelian groups.

\begin{Def}
An exponential monad $\mathbb{M}: \mathcal{C}\to \mathcal{C}$ is called a graded exponential monad if there are functors $\mathbb{M}[k]:\mathcal{C}\to \mathcal{C}$ such that $\mathbb{M}\simeq \bigoplus_k \mathbb{M}[k]$, further there are natural transformations $\mathbb{M}[k]\circ \mathbb{M}[l]\to \mathbb{M}[kl]$ such that the diagram
\xymat{\mathbb{M}[k]\circ \mathbb{M}[l] \ar[r] \ar[d] &\mathbb{M}[kl]\ar[d] \\ \mathbb{M}\circ \mathbb{M} \ar[r] &\mathbb{M}}
commutes for all $k,l$, and the unit $id_{\mathcal{C}}\to \mathbb{M}$ factors as 
\begin{align*}
id_{\mathcal{C}}\to \mathbb{M}[1] \to  \mathbb{M}
\end{align*}
and the augmentation $\mathbb{M}\to id_{\mathcal{C}}$ factors as
\begin{align*}
\mathbb{M} \to \mathbb{M}[1] \to id_{\mathcal{C}}.
\end{align*}
Further there are structure maps $\mathbb{M}[k](X)\otimes \mathbb{M}[l](Y) \to \mathbb{M}[k+l](X\oplus Y)$, such that the following diagram
\xymat{\mathbb{M}[k](X)\otimes \mathbb{M}[l](Y) \ar[r] \ar[d]& \mathbb{M}[k+l](X\oplus Y )\ar[d]\\
\mathbb{M}(X)\otimes \mathbb{M}(Y) \ar[r]^\zeta & \mathbb{M}(X\oplus Y )}
commutes. Further the map $\nu$ should factor as 
\begin{align*}
\mathbbm{1}\to \mathbb{M}[0](0) \to \mathbb{M}(0).
\end{align*}
\end{Def}

\begin{Ex}
If $\mathcal{O}$ is a model for the $\mathbb{E}_\infty$ operad in $E-mod$ for a commutative ring spectrum $E$, then the associated monad $\mathcal{F}_{\mathcal{O}}(A)$, is a graded exponential monad when viewed as a monad on the homotopy category, with the grading given by the arity of the operad.
\end{Ex}

If $\mathbb{M}$ is graded exponential, then this induces a filtration on compositons of $\mathbb{M}$ with itself, given inductively as
\begin{align*}
\mathbb{M}^{\circ q}[k] (X):=\bigoplus_{k=\sum_j j\cdot k_j}\left(\bigotimes_{j\geq 0}\mathbb{M}[k_j]\big( \mathbb{M}^{\circ (q-1)}[j](X)\big) \right) 
\end{align*}
 This induces a filtration of the monadic bar construction on $\mathbb{M}$. We denotes the filtrations $B(1,\mathbb{M},A)[\leq k]$, and the quotients as $B(1,\mathbb{M},A)[=k]$. See \cite{Rez12} for a full discussion.

\begin{Rem} \label{RemOL}
If $A$ is an $\mathbb{M}$ algebra then $B(1,\mathbb{M},\overline{A})=\bigoplus_k B(1,\mathbb{M},A)[=k]$.
\end{Rem}

\begin{Rem}
If $\mathbb{M}:R-mod \to R-mod$ is an exponentially graded monad, and $A$ is a $\mathbb{M}$-algebra there is a spectral sequence:
\begin{align*}
\bigoplus_k H^*\mathrm{Hom}_R(B(1,\mathbb{M},A)[=k],R) \Rightarrow AQ^*_\mathbb{M}(A)
\end{align*} 
natural in both $A$, and $\mathbb{M}$. We will call this spectral sequence the $\mathbb{M}$ spectral sequence.
\end{Rem}
This spectral sequence arises as the dual filtration to the filtration of $B(1,\mathbb{M},A)$ coming from the exponential grading on $\mathbb{M}$.

 Let $E_h$ is the Lubin-Tate theory at height $h$, and $K(h)$ be the height $h$ Morava $K$-theory. We will be concerned with the commutative operad in the $K(h)$-local category, as well at in the category of $K(h)$-local $E_h$-modules. 
 
 \begin{Def}
 If $E$ is a ring spectrum let $E-mod$ denote the category of $E$-modules. If $K$ is a ring spectrum such that $E$ is $K$-local then let $E-mod_K$ be the category of $K$-completed $E$-modules.
 \end{Def}
 
\begin{Def}
Define the (reduced) commutative operads by
\begin{align*}
\mathrm{comm}(n)=\left\{\begin{array}{cc}
(S)_{K(h)} & n>0 \\
* & n=0
\end{array} \right. \\
\mathrm{comm}^E(n)=\left\{\begin{array}{cc}
E_h & n>0 \\
* & n=0
\end{array} \right.
\end{align*} Where $S$ is the sphere spectrum.
For a commutative ring $R$, let
\begin{align*}
\mathrm{comm}^{\mathrm{alg}}(n)=\left\{\begin{array}{cc}
R & n>0 \\
0 & n=0
\end{array} \right.
\end{align*}
Lastly we have the following operad in pointed spaces
\begin{align*}
\mathrm{comm}^{Top}(n)=\left\{\begin{array}{cc}
S^0 & n>0 \\
* & n=0
\end{array} \right.
\end{align*} 
\end{Def}

\begin{Not}
For both topological Andr\'{e}-Quillen cohomology and Andr\'{e}-Quillen cohomology we are going to suppress the operad (respectively the monad) from the notation when it is the commutative operad (respectively the free commutative algebra monad).
\end{Not}

\begin{Def}
Define $\mathrm{sLie}$ to be the Koszul dual operad to commutative operad in either the algebraic or topological settings.

If $M$ is a $R$-module, we define the free shifted Lie-algebra on $M$ with divided powers to be
\begin{align*}
\mathrm{sLie}^{dp}(M):= \prod \Big(\mathcal{K}(\mathrm{Comm}^\mathrm{alg})(n)\otimes M^{\otimes n} \Big)^{\Sigma_n}.
\end{align*} 
\end{Def}

\begin{Rem}
A $R$-chain complex $M$ is an $\mathcal{K}(\mathrm{Comm}^\mathrm{alg})$-algebra, i.e., a $\mathrm{sLie}$-algebra if and only if $M[1]$ is a differential graded Lie algebra, where $M[1]_n=M_{n+1}$.

Further if $R=\mathbb{F}_p$, then from \cite{Fre00} we see that $\mathrm{sLie}^{dp}(M)$ is the free restricted $p$ Lie algebra on $M$. 
\end{Rem}

\subsection{$\nu_h$-periodic unstable homotopy theory}
From now on all our spaces and spectra are completed at an odd prime $p$.
We will in this section summarize the results necessary to carry out our program.

Fix $h$, and let $T$ denote the telescope of a $\nu_h$-self map on a type $h$-complex. Bousfield and Kuhn (see for example \cite{Kuh08}) constructed a functor 
\begin{align*}
\Phi_T: Top_*\to Sp,
\end{align*}
such that $\Phi_T(\Omega^\infty E)\simeq L_TE$ for any spectrum $E$. The homotopy groups of the Bousfield-Kuhn fuctor give a version of the unstable $\nu_h$-periodic homotopy groups of a space:
\begin{Def}
Let $\Phi_h(X):=\Phi_T(X)_{K(h)}$, and $X$ a space.
\begin{align*}
\nu_h^{-1}\pi_*(X):=\pi_*\Phi_h(X).
\end{align*} 
 \end{Def}
\begin{Rem}
Note that this is what in \cite{BeRe17} are called the completed $\nu_h$-periodic homotopy groups, to distinguish them from the ``uncompleted'' unstable $\nu_h$-periodic homotopy groups studied by Bousfield, Davis, Mahowald, and others. These are given as the homotopy groups of the $n$'th telescopic monochromatic layer of $\Phi_T$.
\end{Rem}

For $X$ a finite space, and $E$ a commutative ring spectrum, we let $E^X$ be the Spanier-Whitehead dual of $X$ in $E$-modules as a non-unital commutative $E$-algebra. In \cite{BeRe17}, Behrens and Rezk constructed a map 
\begin{align*}
c: \Phi_h(X)\to  B(1,\mathcal{F}_{\mathrm{comm}}, S_{K(h)}^X)^\vee,
\end{align*}
where $\mathcal{F}_{\mathrm{comm}}$ is taken in the $K(h)$-local category, i.e., we $K(h)$-localize all coproducts.
\begin{Def}[\cite{BeRe17Survey}]
We say that a space $X$ is $\Phi_h$-good if the map $c$ is an equivalence for $X$.
\end{Def}

\subsection{Goodwillie Calculus}
 Given a functor $F: Top_* \to \mathcal{C}$, where $\mathcal{C}$ is the category spaces, spectra or some localization of spectra. Assume that $F$ preserves weak equivalences, $F$ is finitary, (i.e., determined by its value on finite CW-complexes) and $F(*)\simeq *$.  Goodwillie in \cite{Goo03} constructed a tower of functors $P_n(F):Top_*\to \mathcal{C}$, under $F$:
\xymat{F(X)\ar[dr] \ar[drr] && \\ \ldots \ar[r]& P_2(F)(X) \ar[r] &P_1(F)(X)}
Under certain conditions on both $F$ and $X$, one gets an equivalence 
\begin{align*}
F(X) \simeq \varprojlim_n P_n(F)(X).
\end{align*}
Further the layers of the tower 
\begin{align*}
D_n(F)(X):= \mathrm{Fib}\big[ P_n(F)(X)\to P_{n-1}(F)(X)\big] 
\end{align*} have the form $\Omega^\infty_\mathcal{C} (\partial_n(F) \wedge_{h\Sigma_n} X^{\wedge n})$, for some Borel-$\Sigma_n$-equivariant spectrum $\partial_n(F)$, and $\Omega^\infty_\mathcal{C}$ is $\Omega^\infty$ if $\mathcal{C}=Top_*$, and the identity if $\mathcal{C}$ is spectra. This implies that when the tower converges, i.e., when $F(X) \simeq \varprojlim_n P_n(F)(X)$, then we get a spectral sequence computing $\pi_*F(X)$ with input only dependent on stable information. 

When $F=id_{Top_*}$
 then $P_1(id_{Top_*})(X)=\Omega^\infty \Sigma^\infty X$,
  and $\pi_*P_2(id_{Top_*})(X)$ are the metastable homotopy groups, see \cite{Mah67}. The 
tower thus filters the homotopy groups, starting with some classical notions. Further 
it was shown in \cite{Chi05} that $\partial_*(id_{Top_*}):=\{\partial_n(id_{Top_*})\}_n$
 forms an operad. In the same paper  it was shown that $\partial_*(id_{Top_*})\simeq \mathrm{sLie}$, as operads. 

If we think of $\Phi_h$ as being a functor from spaces to $K(h)$-local spectra, then $\Phi_h$ is finitary, and thus we can set up Goodwillie calculus for it. 
It was proven in \cite{BeRe17Survey} that $P_n(\Phi_h)\simeq \Phi_h \circ P_n$. This implies that when the Goodwillie tower converges, we can calculate unstable $\nu_h$-periodic homotopy groups from stable $\nu_h$-periodic homotopy groups. 

Further from \cite{BeRe17} the map $c$ induces an equivalence
\begin{align} \label{P_n(Phi)}
P_k(\Phi_h)(X)\stackrel{\sim}{\to}  B(1,\mathcal{F}_{\mathrm{comm}}, S_{K(h)}^X)[\leq k]^\vee,
\end{align}
where again $\mathcal{F}_{\mathrm{comm}}$ is taken in the $K(h)$-local category.
Though in general convergence fails for the Taylor tower $P_k(\Phi_h)$, see \cite{BrHe16}, and therefore $c$ is in general not an equivalence. The result of \cite{AM99} can be interpreted to say that this map is an equivalence when $X$ is a sphere, and in \cite{BeRe17Survey} they gave a condition on a space being build by spherical fibrations implying that it is $\Phi_h$-good. 

\section{The monad $\mathbb{T}$}
 Recall, from \cite{Rez09}, that there is a monad 
\begin{align*}
 \mathbb{T}: (E_h)_*-mod \to (E_h)_*-mod
\end{align*} 
such that 
for $X\in E_h-mod$, with $\pi_* X$ is flat as a $(E_h)_*$-module, then
\begin{align*}
\pi_*\big(\mathcal{F}_{\mathrm{comm}^E}X\big)_{K(h)}\simeq \mathbb{T}\pi_*X_{\hat{\mathfrak{m}}},
\end{align*}
 Where for $E_*\cong \mathbb{Z}_{p^n}[[u_1,\ldots , u_{n-1}]][u^{\pm 1}]$, then $\mathfrak{m}=(p,u_1,\ldots , u_{n-1},u)$. Note that this differs from Rezk's notation where the monad he calls $\mathbb{T}$ has the property that $\pi_*\mathcal{F}_{\mathrm{comm}_+}X_+\simeq \mathbb{T}\pi_*X_+$, where $(\ )_+$ denotes a disjoint basepoint, and where $\mathrm{comm}_+$ is the the non reduced commutative operad, i.e., $\mathrm{comm}_+(0)=E_h$.
 
 \begin{Lem} If $X$ is a space such that $E_h^*X$ is a finitely generated free $E_h^*$-algebra, then there is a spectral sequence $AQ^*_\mathbb{T}(E_h^*X)\Rightarrow TAQ^*(E_h^X)$.
 \end{Lem}
 \begin{proof}
This is Proposition 4.7 of \cite{BeRe17}.
 \end{proof}
In \cite{BeRe17} this spectral sequence is called the Basterra spectral sequence. 

\begin{Rem} \label{BastSSrel} From \cite{Rez09} we know that
$\mathbb{T}$ inherits an exponential grading such that $X\in E_h-mod$, such that $\pi_* X$ is finite and flat as a $(E_h)_*$-module, then 
\begin{align*}
\pi_*\big(\mathcal{F}_{\mathrm{comm}^E}[ k]X\big)_{K(h)}\simeq \mathbb{T}[ k]\pi_*X.
\end{align*}
\end{Rem}
Note that the Basterra spectral sequence is well behaved with respect to the exponential grading of the monads $\mathcal{F}_{\mathrm{comm}^E}$ and $\mathbb{T}$, and hence we easily see 
 \begin{Cor}\label{BastSS}
 If $X$ is a space such that $E_h^*X$ is a finitely generated free $E_h^*$-algebra, then for all $k$ there are spectral sequences with maps between them:
\xymat{H^*(B(1,\mathbb{T},E_h^*X)[=k]) \ar[d] \ar@{=>}[r] & \pi_*(B(1,\mathcal{F}_{\mathrm{comm}^E},E_h^X)[=k])^\vee \ar[d] \\
H^*(B(1,\mathbb{T},E_h^*X)[\leq k])  \ar@{=>}[r] & \pi_*(B(1,\mathcal{F}_{\mathrm{comm}^E},E_h^X)[\leq k])^\vee  \\
AQ^*_\mathbb{T}(E_h^*X) \ar[u]  \ar@{=>}[r] & TAQ^*(E_h^X) \ar[u]  
}
\end{Cor}

\begin{Not}
We will exclusively focus on the height $h=1$ case. 
 Let $E=E_1$ be the spectum representing $p$-completed complex $K$-theory, and $E^*$, $E_*$ be its cohomology, homology respectively. Further let $\Phi:=\Phi_1$.
\end{Not}

From \cite{McC83} we see that if $A$ is a free $E_*$-module with a finite basis $\{a_1,\ldots , a_n\}$ then $\mathbb{T}(A)$ is the free commutative $E_*$-algebra with basis $\bigcup_{i=0}^\infty\{\theta^i a_1,\ldots \theta^i a_n\}$, where we identify $\theta^0a=a$. Further we have $a_i\in \mathbb{T}[1](A)$, $\theta^i(a_i)\in \mathbb{T}[p^i](A)$, and if $x\in \mathbb{T}[k](A)$, and $y\in \mathbb{T}[l](A)$ then $xy\in \mathbb{T}[k+l](A)$.

\section{Monad cohomology for $\mathbb{T}$}
\begin{Prop}\label{GrothCollapse} If $X$ is a finite space such that $E^*X$ is a finitely generated and free $E^*$-algebra, the spectral sequence from Corollary \ref{BastSSrel},
\begin{align*}
H^*(B(1,\mathbb{T},E^*X)[=k])\Rightarrow \pi_*(B(1,\mathcal{F}_{\mathrm{comm}^E},E^X)[=k])^\vee,
\end{align*}
 collapses for all $k$. Further we see that 
 \begin{align*}
 \bigoplus_k H^*(B(1,\mathbb{T},E^*X)[=k])\simeq \Lambda(\overline{\theta})\otimes \mathrm{sLie}^{dp}(E_*X). 
 \end{align*}
 Here $\mathrm{sLie}^{dp}$ is the free $E_*$-Lie algebra with divided powers, and $\overline{\theta}$ is of degree $p$ with respect to the exponential grading of $\mathbb{T}$, and the homological degree $|\overline{\theta}(x)|=p|x|-1$.
\end{Prop}
Before we start the proof we need to recall a result of \cite{Bra17}.
\begin{Theo}\label{Brantner}
If $X$ is a finite space such that $E^*X$ is a finitely generated and free $E^*$-algebra, then the $E_1$-page of the $E$-based Goodwillie spectral sequence for $\Phi$, at $X$ is given by $\Lambda(\theta)\otimes \mathrm{sLie}(E_*X)$.
\end{Theo}
\begin{proof}
This follows immediately from \cite{Bra17} Theorem 4.4.4.
\end{proof}
\begin{Rem}
We will later define an operation $\overline{\theta}$ on $AQ_\mathbb{T}^*(A)$, though it is not clear to the author how this operation relates to the operation $\theta$ from Brantner's thesis.
\end{Rem}
\begin{proof}[Proof of Proposition \ref{GrothCollapse}]
Due to Remark \ref{RemOL} it is enough to show that
\begin{align*}
 AQ^*_\mathbb{T}(\overline{E^*X})\Rightarrow \pi_*TAQ(\overline{E^X})
\end{align*}
 collapses.
Recall from \cite{BeRe17}, Proposition 3.5, that there is a Grothendieck spectral sequence 
\begin{align} \label{GrothSS}
E_2^{s,t}=\mathrm{Ext}_{\Delta}^s( AQ_t(\overline{E^*X}), \overline{E^*})\Rightarrow AQ^{s+t}_\mathbb{T}(\overline{E^*(X)})
\end{align}
since $\overline{E^*X}$ is free as a $E^*$-module.

Recall that $AQ_*$ is the same as the operad homology for the algebraic commutative operad $\mathrm{comm}^{alg}$, see \cite{LoVa12} section 12.1.1.  
From the same section it follows that 
\begin{align*}
AQ_*(\overline{E^*X})\cong \bigoplus_k \big(\mathrm{sLie}(k)\otimes E^*(X)^{\otimes k}\big)_{\Sigma_k}.
\end{align*}

So it follows that the input to the Grothendieck spectral sequence is 
\begin{align*}
\mathrm{Ext}_{\Delta}^*( AQ_*(\overline{E^*X}), \overline{E^*})\simeq \Lambda(\overline{\theta})\otimes \mathrm{sLie}^{dp}(E_*X)
\end{align*}
as $\Delta=\mathbb{Z}_p[\theta]$.
Note that in the spectral sequence from (\ref{GrothSS}) is concentrated in degree $s=0,1$ as $\Delta= \mathbb{Z}_p[\theta ]$. The $d_2$-differential is $E_2^{s,t} \to E_2^{s+2,t-1}$, therefore we can conclude that the spectral sequence collapses and therefore 
\begin{align*}
H^* B(1,\mathbb{T}, \overline{E^*(X)})\simeq\Lambda(\overline{\theta})\otimes \mathrm{sLie}^{dp}(E_*X).
\end{align*}

We can see from the equivalence $(\ref{P_n(Phi)})$ that the $E$-based spectral sequence for $\Phi$ at $X$ coincides with the $TAQ$ spectral sequence for $E^X$. We therefore have from Theorem \ref{Brantner} know that the abutment of the spectral sequence is $\Lambda(\theta)\otimes \mathrm{sLie}(E_*X)$. It is clear that both input and output of the spectral sequence of the statement of the proposition is free, and after rationalization has the same dimension over $\mathbb{Q}_p[u^{\pm 1}]$, where $E_*=\mathbb{Z}_p[u^{\pm 1}]$. Thus there is no room for differentials, and that concludes the proof.
\end{proof}

\begin{Cor} \label{AQ triv}
If $M=E_*\{x_1,\ldots x_n\}$ is a trivial $\mathbb{T}$-algebra, then 
\begin{align*}
AQ_\mathbb{T}^*(M)\simeq \Lambda(\overline{\theta})\otimes \mathrm{sLie}^{dp}(M^\vee).
\end{align*}
\end{Cor}

\subsection{The operation $\overline{\theta}$}
We will now give a chain level description of the $\overline{\theta}$ that showed up in Proposition \ref{GrothCollapse}.
Assume that $M$ is free as a $E_*$-module with generators $x_1,\ldots x_n$, then $\mathbb{T}(M)\cong E_*[\theta^ix_j|i\geq 0, \ 1\leq 1\leq n]$. Given a monomial $m\in \mathbb{T}(E_*\{a_1,\ldots a_k\})$, of the form $m=\big(\theta^{i_1}a_{1} \big)^{e_1}\cdot \ldots \cdot \big(\theta^{i_k}a_{j_k} \big)^{e_k}$, we can define an operation
\begin{align*}
\underline{m}: M^{\otimes k}&\to \mathbb{T}(M) \\
x_{j_1}\otimes \ldots \otimes  x_{j_k} &\mapsto \big(\theta^{i_1}x_{j_1} \big)^{e_1}\cdot \ldots \cdot \big(\theta^{i_k}x_{j_k} \big)^{e_k}
\end{align*}
We can therefore write elements of $\mathbb{T}(M)$ as sums of elements of the form
\begin{align*}
 \underline{m}\left| \begin{array}{c}
x_{j_1} \\ \vdots \\ x_{j_k} 
\end{array} \right.  = \underline{m}|\overrightarrow{x}
\end{align*}
For the element $\underline{m}(x_{j_1}\otimes \ldots x_{j_k})$. There is no need for us to restrict this definition to monomials, instead of all polynomials.
 Note if $m^1,m^2_1, \ldots m^2_k$ are monomials of this form we get elements of $\mathbb{T}(\mathbb{T}(M))$ by
\begin{align*}
 \underline{m_1}\left| \begin{array}{c}
\underline{m_{1}^2} \\ \vdots \\ \underline{m^2_{k}}  
\end{array} \right|   \begin{array}{c}
\overrightarrow{x_{1}} \\ \vdots \\ \overrightarrow{x_{k}} 
\end{array} = \underline{m^1}\left| \overrightarrow{\underline{m^2}}\right|\overrightarrow{x}
\end{align*}
and we see that these elements generate all of $\mathbb{T}^{\circ 2}(M)$. In the same way we can construct generators $\mathbb{T}^{\circ s}(M)$ with the names
\begin{align*} \begin{array}{c|c|c|c|c}
\underline{m^1} & \overrightarrow{\underline{m^2}}& \ldots & \overrightarrow{\underline{m^s}}  &\overrightarrow{x}.
\end{array}
\end{align*}
Let $\overline{B}$ denote the normalized bar complex, then we wish to define
\begin{align*}
\Theta^\vee: \overline{B}_s(1,\mathbb{T},A) &\to \overline{B}_{s-1}(1,\mathbb{T},A) \\ 
\left[\begin{array}{c|c|c|c|c}
\underline{m^1} & \overrightarrow{\underline{m^2}}& \ldots & \overrightarrow{\underline{m^s}}  &\overrightarrow{x}
\end{array} \right] &\mapsto \left\{ \begin{array}{cc}(-1)^s \left[\begin{array}{c|c|c|c}
 \overrightarrow{\underline{m^2}}& \ldots & \overrightarrow{\underline{m^s}}  &\overrightarrow{x}
\end{array} \right] & \text{if } m^1=\theta (a_1), \\ \\
0 & \text{else.}
\end{array} \right.
\end{align*}
We wish to check that this in fact forms a chain map. The first case is $s=1$, if $m^1\neq \theta(a)$ then it is clear, otherwise we see that
\xymat{[\theta | x_j]\ar@{|->}[r]^{\Theta^\vee} \ar@{|->}[d]^d & -x_j \ar@{|->}[d]^d \\ -\theta(x_j) \ar@{|->}[r]& 0.}
For $s>1$, and $m^1=\theta(a)$, we have
\begin{align*}
d\circ \Theta^\vee \Big(\left[ \underline{\theta(a)} \left| \underline{m^2}\right| \ldots \left| \overrightarrow{\underline{m^s}}  \right|\overrightarrow{x} \right] \Big) &= (-1)^{s-1} d\left.\left(\left[ \underline{m^2}\right| \ldots \left| \overrightarrow{\underline{m^s}}  \right|\overrightarrow{x} \right]\right) \\
&=\Theta^\vee \circ d \Big( \left[ \underline{\theta(a)} \left| \underline{m^2}\right| \ldots \left| \overrightarrow{\underline{m^s}}  \right|\overrightarrow{x} \right]\Big)
\end{align*}
Since $m^2\neq 1$ then it is clear that the term $\left[-\underline{\theta(a)}\circ \underline{m^2}\right| \ldots \left| \overrightarrow{\underline{m^s}}\left.  \right|\overrightarrow{x} \right]$ get send to zero by $\Theta^\vee$, and $(-1)^{s+1}=(-1)^{s-1}$. When $m^1\neq \theta(a_1)$, then we know that $\underline{m^1}( m^2)$ does not contain any terms of the form $\theta(a_j)$, and hence we are done. We can therefore give the following definition.
\begin{Def}\label{DefTheta}
If $A$ is a $\mathbb{T}$-algebra then we have operations:
\begin{align*}
\Theta: AQ_\mathbb{T}^*(A)  \to AQ^{*+1}_\mathbb{T}(A)
\end{align*}
given by $H^*(\Theta^\vee)$.
\end{Def}

\begin{Lem} \label{Theta_def}
If $A$ is a finitely generated and free as a $E^*$-module, then acting by $\Theta$ on $AQ^*_\mathbb{T}(\overline{A})$ coincides with multiplication by $\overline{\theta}$ in  $\Lambda(\overline{\theta})\otimes \mathrm{sLie}^{dp}(A^\vee)$. 

Further we can restrict $\Theta$ to
\begin{align*}
\Theta: H^*B(1,\mathbb{T},A)[\leq k])& \to H^{*+1}B(1,\mathbb{T},A)[\leq pk]) \\ 
\Theta: H^*B(1,\mathbb{T},A)[= k])& \to H^{*+1}B(1,\mathbb{T},A)[=pk])
\end{align*}
\end{Lem}
\begin{proof}
The first statement follows from checking the Basterra spectral sequence, and the second sattement is an easy check.
\end{proof}

\subsection{The bracket, and the elements of $AQ_\mathbb{T}$}
One would hope to show that $AQ^*_\mathbb{T}$ admits the structure of a shifted Lie algebra with divided powers, unfortunately this is not clear to the author at this stage how to do this. So instead we will mimic the construction from the previous subsection.

Note that for $a_1a_2\in \mathbb{T}(E_*\{a_1a_2\})$, then if $m^1=a_1a_2$,
we can write
\begin{align*}
\left[\begin{array}{c|c|c|c|c}
\underline{m^1} & \overrightarrow{\underline{m^2}}& \ldots & \overrightarrow{\underline{m^s}}  &\overrightarrow{x}
\end{array} \right] = \left[ \begin{array}{c|c|c|c|c}
\underline{a_1a_2} & \begin{array}{c}
\overrightarrow{\underline{m^2}}_1 \\ \overrightarrow{\underline{m^2}}_2
\end{array} & \ldots & \begin{array}{c}
\overrightarrow{\underline{m^s}}_1 \\ \overrightarrow{\underline{m^s}}_2 
\end{array}   & \begin{array}{c}
\overrightarrow{x}_1 \\ \overrightarrow{x}_2 
\end{array}
\end{array} \right]
\end{align*}
If $x\in A$ we define 
\begin{align*}
\{\ ,x\}^\vee:\overline{B}_s(1,\mathbb{T},A)\to\overline{B}_{s-1}(1,\mathbb{T},A)
\end{align*}
by
\begin{align*}
\left[\begin{array}{c|c|c|c|c}
\underline{m^1} & \overrightarrow{\underline{m^2}}& \ldots & \overrightarrow{\underline{m^s}}  &\overrightarrow{x}
\end{array} \right] \mapsto (-1)^{s-1} \left[\begin{array}{c|c|c|c}
 \overrightarrow{\underline{n^2}}& \ldots & \overrightarrow{\underline{n^s}}  &\overrightarrow{y}
 \end{array} \right]
\end{align*}
if 
\begin{align*}
\left[\begin{array}{c|c|c|c|c}
\underline{m^1} & \overrightarrow{\underline{m^2}}& \ldots & \overrightarrow{\underline{m^s}}  &\overrightarrow{x}
\end{array} \right] = \left[ \begin{array}{c|c|c|c|c}
\underline{a_1a_2} & \begin{array}{c}
\overrightarrow{\underline{n^2}} \\ \underline{1}
\end{array} & \ldots & \begin{array}{c}
\overrightarrow{\underline{n^s}} \\ \underline{1} 
\end{array}   & \begin{array}{c}
\overrightarrow{y} \\ x 
\end{array}
\end{array} \right],
\end{align*}
and $0$ else.
Here $1:A\to \mathbb{T}(A)$ sends $a\mapsto a$.
Completely analogous to the definition of $\Theta^\vee$ above we easily check that this commutes with the differentials, and thus define a chain map.

\begin{Def} \label{DefBracket}
If $A$ is a $\mathbb{T}$-algebra, and $x\in A$, such that $x^\vee\in AQ^0_\mathbb{T}(A)$, then we can define
\begin{align*}
\{\ ,x^\vee\}: AQ^*_\mathbb{T}(A) \to AQ^{*+1}_\mathbb{T}(A) 
\end{align*}
by $H^*(\{\ ,x\}^\vee)$

Further for $y\in AQ^*_\mathbb{T}(A)$ define $[y,x^\vee]:=S_{y,x}\cdot \{\ ,x^\vee\}(y)$, where 
\begin{align*}
S_{y,x}:= \left\{\begin{array}{cc}
2 &\text{if } y=x^\vee, \\
1 & \text{Otherwise.}
\end{array} \right.
\end{align*}
\end{Def}

\begin{Rem}\label{bracketdef}
If $A$ is a commutative $R$ algebra, then it follows from \cite{LoVa12} section 12.1.1 that $AQ^*(A)$ admits the structure of a shifted Lie algebra with divided powers. Further the we have a map of exponentially graded monads $\mathcal{F}_{\mathrm{Comm}^{Alg}}\to \mathbb{T}$, further it is easy to check that the following diagram commutes. Let $\{a_1,\ldots, a_n\}$ be a $E_*$ generating set for $AQ^0_\mathbb{T}(A)$ then
\xymat{ AQ_\mathbb{T}^*(A)\otimes E_*\{a_1,\ldots , a_n\} \ar[rrr]^{[\ ,\ ]} \ar[d] &&& AQ_\mathbb{T}^*(A)\ar[dd] \\  AQ_\mathbb{T}^*(A)\otimes AQ_\mathbb{T}^*(A) \ar[d]&&& \\ AQ^*(A)\otimes AQ^*(A) \ar[rrr]^{[\ ,\ ]} && &AQ^*(A).}
Where the above horizontal map is the bracket as defined above, and the lower by the bracket in the shifted Lie algebra with divided powers.
\end{Rem}

\begin{Lem}\label{bracket}
Under the identification of $AQ^*_\mathbb{T}(\overline{A})\simeq \Lambda(\overline{\theta})\otimes \mathrm{sLie}^{dp}(A^\vee)$, given $y \in 1\otimes \mathrm{sLie}^{dp}(A^\vee)$, and $x\in A^\vee$ then $[y,x]$ as defined above, coincides with $[y,x]\in \mathrm{sLie}(A^\vee)$.

Further for $x\in A$ such that $x^\vee\in AQ^0_\mathbb{T}(A)$ the bracket restricts to
\begin{align*}
 H^*(B(1,\mathbb{T},A)[\leq k])  \stackrel{[\ ,x^\vee ]}{\to} H^{*+1}(B(1,\mathbb{T},A)[\leq k+1]) \\ 
 H^*(B(1,\mathbb{T},A)[= k])   \stackrel{[\ ,x^\vee ]}{\to}H^{*+1}(B(1,\mathbb{T},A)[= k+1] )
\end{align*}
\end{Lem}
\begin{proof}
The first statement follow from simply checking the identification of Corollary \ref{AQ triv}. The second statement follows from easily checking the definitions.
\end{proof}

\subsection{Differentials in the $\mathbb{T}$-spectral sequence}

\begin{Lem} \label{AlgLeibniz}
Let $A$ be a $\mathbb{T}$-algbera, then in particular it is a commutative $E^*$-algebra. Assume that $A$ is free as a commutative $E^*$-algebra, and finitely generated as a free $E^*$-module. Then the $E_2$ page of the $\mathbb{T}$-spectral sequence is equivalent to $Q(A)\oplus \overline{\theta}\mathrm{Lie}^{dp}(A^\vee)$, where $Q(A)$ is the indecomposables of $A$ as commutative algebra.
\end{Lem}
\begin{proof}
Note that we have a map of graded exponential monads $\mathcal{F}_{\mathrm{Com}^{\mathrm{Alg}}}\to \mathbb{T}$. This induces a map of spectral sequences:
\xymat{\Lambda(\overline{\theta})\otimes \mathrm{sLie}^{dp}(A^\vee) \ar[d] \ar@{=>}[r]& AQ^*_\mathbb{T}(A) \ar[d] \\
\mathrm{sLie}^{dp}(A^\vee)  \ar@{=>}[r]& AQ^*(A).}
Here the lower spectral sequence is the $\mathrm{Comm}$-spectral sequence, which we know to converge to $Q(A)$. We further know that it collapses on the $E_2$-page, \cite{LoVa12} Proposition 12.1.1. 

 Using Lemma \ref{bracket} and Remark \ref{bracketdef} it is easy to see that all $d_1$'s in the $\mathbb{T}$-spectral sequence are exactly the same as the $d_1$'s of the $AQ$ spectral sequence. That proves the statement.
\end{proof}

\begin{Lem} \label{Theta inj}
Let $A$ be as in the above lemma. Then we know that $A\cong E_*[x_1,\ldots x_n]$, assume that $\theta^\vee: E^*\{x_1^\vee,\ldots , x_n^\vee\}\to A^\vee$ is injective, where $\theta^\vee$ is the dual to $\theta$. Further assume that $\theta^\vee(x_i^\vee)=\sum_j\lambda_j x_j^\vee$, for $\lambda\in E_*$. Then we see that 
\begin{align*}
d_{p-1}(x_i^\vee)=\sum_j -\lambda_j\overline{\theta}x^\vee_j.
\end{align*}
\end{Lem}
\begin{proof}
According to Lemma \ref{Theta_def} we have 
\begin{align*}
\overline{\theta}x_i^\vee=[\theta(a_1)| (x_i)]\in H^*(B(1,\mathbb{T},A)[=p]).
\end{align*}
In $B(1,\mathbb{T},A)$, as a chain complex, we have $d(\theta(a_1) | (x))=-\theta(x)$, this concludes the proof.
\end{proof}

We now wish to prove the last family of differentials in our $\mathbb{T}$-spectral sequence:
\begin{Lem}\label{theta-d_1}
If $A$ is as in Lemma \ref{Theta inj}, assume that $d_1(x)=y$ in the associated $\mathbb{T}$-SS, then $d_p\overline{\theta}x=\overline{\theta}y$.
\end{Lem}
 \begin{proof}  
 This follows easily from us constructing $\Theta^\vee$ as a chain level operation above.
 \end{proof}

\begin{Cor} \label{E_inftyT}
If $A$ is as in Lemma \ref{Theta inj} then the $E_\infty$-page of the $\mathbb{T}$-spectral sequence for $A$ converges to $Q(A)/\theta $.
\end{Cor}
\begin{proof}
This follows easily from running the above differentials.
\end{proof}

\section{Operad cohomology for $\mathrm{Comm}^E$}

We will use Corollary \ref{BastSS} to leverage the differentials in the $\mathbb{T}$ spectral sequence we know, to get knowledge about the differentials of the TAQ spectral sequence. 

\begin{Theo}
Assume that $E^*X$ is free as a commutative algebra, and free and finitely generated as a $E^*$-module. This implies that 
\begin{align*}
E^*(X)=E^*[x_1,\ldots ,x_n],
\end{align*}
where $x_i$ is in an odd degree. Assume further that the $\mathbb{T}$-algebra structure is such that Assume that there is assume that $\theta^\vee: E^*\{x_1^\vee,\ldots , x_n^\vee\}\to E^*(X)$ is injective.  Then the Basterra spectral sequence 
\begin{align} \label{BastStunt}
H^*B(1,\mathbb{T}, E^*X)[\leq k]\Rightarrow \pi_*B(1,\mathcal{F}_{\mathrm{comm}^E}, E^X)[\leq k]^\vee
\end{align}
 collapses for all $k$.
\end{Theo}
\begin{proof}
We will proceed by induction on $k$

From Proposition \ref{GrothCollapse} we know that  
\begin{align*}
H^*B(1,\mathbb{T}, E^*X)[= l]\Rightarrow \pi_*B(1,\mathcal{F}_{\mathrm{comm}^E}, E^X)[= l]^\vee
\end{align*}
collapses for all $l$. This gives the statement of $(\ref{BastStunt})$ for $k=1$.

\emph{Step 1.} Assume now that we have statement $(\ref{BastStunt})$ for all $1\leq k < p-1$. The following diagram
\xymat{H^*(B(1,\mathbb{T},E^*X)[=k]) \ar[d] \ar@{=>}[r]^\sim & \pi_*B(1,\mathcal{F}_{\mathrm{comm}^E},E^X)[=k])^\vee \ar[d] \\
H^*(B(1,\mathbb{T},E^*X)[\leq k])  \ar@{=>}[r]^\sim & \pi_*B(1,\mathcal{F}_{\mathrm{comm}^E},E^X)[\leq k])^\vee  \\
H^*(B(1,\mathbb{T},E^*X)[\leq k]) \ar[u] \ar[d] \ar@{=>}[r]^\sim & \pi_*B(1,\mathcal{F}_{\mathrm{comm}^E},E^X)[\leq k])^\vee \ar[u] \ar[d] \\
H^{*+1}(B(1,\mathbb{T},E^*X)[=(k+1)]) \ar@{=>}[r]^\sim & \pi_{*+1}B(1,\mathcal{F}_{\mathrm{comm}^E},E^X)[=(k+1])^\vee 
} Shows that we can lift all $d_1$'s starting on the $k$-line of the $\mathbb{T}$-spectral sequence to $d_1$'s starting on the $k$-line of the $TAQ$-spectral sequence. It now follows from Lemma \ref{AlgLeibniz} that $(\ref{BastStunt})$ collapses for $k+1$.

\emph{Step 2.} We now want to show the statement for $k=p$. Note the above argument allows us to lift all $d_1$'s starting on the $(p-1)$-line of the $\mathbb{T}$ spectral sequence to the $TAQ$-spectral sequence. We therefore only need to lift the $d_{p-1}$-differentials from the $\mathbb{T}$-spectral sequence that we found in Lemma \ref{Theta inj}. That this is possible follows from the following diagram:
\xymat{H^*(B(1,\mathbb{T},E^*X)[=1]) \ar[d] \ar@{=>}[r]^\sim & \pi_*B(1,\mathcal{F}_{\mathrm{comm}^E},E^X)[=1])^\vee \ar[d] \\
H^*(B(1,\mathbb{T},E^*X)[=1])  \ar@{=>}[r]^\sim & \pi_*B(1,\mathcal{F}_{\mathrm{comm}^E},E^X)[=1])^\vee  \\
H^*(B(1,\mathbb{T},E^*X)[\leq p-1]) \ar[u] \ar[d] \ar@{=>}[r]^\sim & \pi_*B(1,\mathcal{F}_{\mathrm{comm}^E},E^*X) [\leq p-1])^\vee \ar[u] \ar[d] \\
H^{*+1}(B(1,\mathbb{T},E^*X)[=p]) \ar@{=>}[r]^\sim & \pi_{*+1}B(1,\mathcal{F}_{\mathrm{comm}^E},E^X)[=p])^\vee 
}

Now assume that we have showed $(\ref{BastStunt})$ for all $1\leq k < ip-1$, then by similar arguments to Step 1 above, we can extend the result to $k+1$.  

Assume now that we have showed $(\ref{BastStunt})$ for all $1\leq k \leq ip-1$, then we can again lift the $d_1$'s from previous arguments. So we just need to lift the $d_p$'s from the $\mathbb{T}$-spectral sequence found in Lemma \ref{theta-d_1}. That this can be done follows from the following diagram:
\xymat{H^*(B(1,\mathbb{T},E^*X)[=p(i-1)]) \ar[d] \ar@{=>}[r]^\sim & \pi_*B(1,\mathcal{F}_{\mathrm{comm}^E},E^X)[=p(i-1)])^\vee \ar[d] \\
H^*(B(1,\mathbb{T},E^*X)[\leq p(i-1)])  \ar@{=>}[r]^\sim & \pi_*B(1,\mathcal{F}_{\mathrm{comm}^E},E^X)[=p(i-1)])^\vee  \\
H^*(B(1,\mathbb{T},E^*X)[\leq ip-1]) \ar[u] \ar[d] \ar@{=>}[r]^\sim & \pi_*B(1,\mathcal{F}_{\mathrm{comm}^E},E^X) [\leq ip-1])^\vee \ar[u] \ar[d] \\
H^{*+1}(B(1,\mathbb{T},E^*X)[=ip]) \ar@{=>}[r]^\sim & \pi_{*+1}B(1,\mathcal{F}_{\mathrm{comm}^E},E^X)[=ip])^\vee 
}

This concludes the proof by Corollary \ref{E_inftyT}.
\end{proof}

\begin{Cor} \label{E_inftyTAQ}
If $X$ is as above, then the $E_\infty$-page of the $TAQ$-spectral sequence for $E^X$ converges to $Q(E^*X)^\vee/\theta^\vee$.
\end{Cor}

\section{The $\nu_1$-periodic homotopy groups, and $\Phi_1$-goodness}
Bousfield proves that
\begin{Theo}[ \cite{Bou99} section 9.2] \label{BousThm}
If $X$ is a $1$-connected $H$-space  with $H_*(X;\mathbb{Q})$ associative and with $H_*(X;\mathbb{Z}_{(p)})$ finitely genrated over $\mathbb{Z}_{(p)}$, then
\begin{align*}
\nu_1^{-1}\pi_{2m}(X) & \cong W^m(Q(E^1(X)/\mathrm{im}\theta)^\# \\
\nu_1^{-1}\pi_{2m-1} (X) & \cong W_1^m(Q(E^1(X)/\mathrm{im}\theta)^\#.
\end{align*}
Where when $W^m=\mathrm{coker}(\psi^l-l^m)$, $W^m_1=\mathrm{ker}(\psi^l-l^m)$, where $l\in \mathbb{Z}_p^\times$ is a topological generator, and $(\ )^\#$ is the Pontryagin dual of the group. 
\end{Theo}
\begin{proof}
From \cite{Lin78} we know that the conditions on $X$ implies $E^*X$ is a free finitely generated $E^*$-algebra, and free and finitely generated as a $E^*$-module. This implies that $E^*(X_+)=E^*[x_1,\ldots ,x_{d}]$, where $x_i$ is in an odd degree. Further from the proof of Theorem 6.2 in \cite{Bou99} we see that $\theta(x_i)\in E^*\{x_1,\ldots , x_d\}$, and it is injective due to the proof of Theorem 9.2 in \cite{Bou99}.

From Corollary \ref{E_inftyTAQ} the $E_\infty$-page of the TAQ spectral sequence is isomorphic to $Q(E^*(X))^\vee/ \mathrm{Im}d_{p-1}$, and $d_{p-1}:E_p^{*,1}\to E_p^{*,p}$ is given by $(\theta)^\vee$. 

 So $\theta: Q(E^*(X)\to Q(E^*(X)$ can be represented by an upper triangular matrix, and hence the cokernel of $\theta$ is isomorphic to the cokernel of $(\psi^p)^\vee$. So we get that the $E_\infty$-page of the TAQ spectral sequence is isomorphic to $(Q(E^*(X)/\mathrm{Im}(\theta))^\vee$, clearly this is trivial when $*$ is even, and when $*$ is odd isomorphic to $(Q(E^1(X)/\mathrm{im}\psi^p)^\#$.

Recall that we have a fiber sequence $S_{K(1)} \to E \stackrel{\psi^l-1}{\to}E$, which induces long exact sequences 
\begin{align*}
\ldots \to E_{*+1}(X) \to \nu_1^{-1}\pi_*(X) \to E_*(X) \stackrel{\psi^l-1}{\to} \ldots 
\end{align*}
Recall that $\psi^l$ is a ring homomorphism with $\psi^l(u)=l\cdot u$. So 
\begin{align*}
\psi^l-1: E_{2m-\epsilon}(X) \to E_{2m-\epsilon}(X),
\end{align*}
where $\epsilon =0,1$, is, up to a unit, the same as $\psi^l-l^m: E_{-\epsilon} \to E_{-\epsilon}$ under the identification of $E_{2m-\epsilon} \cong E_{-\epsilon}$. This concludes the proof.
\end{proof}

\begin{Cor}
If $X$ is a $1$-connected $H$-space  with $H_*(X;\mathbb{Q})$ associative and with $H_*(X;\mathbb{Z}_{(p)})$ finitely genrated over $\mathbb{Z}_{(p)}$, then $X$ is $\Phi_1$-good in the sense of \cite{BeRe17Survey}.
\end{Cor}
\begin{proof}
The methods presented in \cite{Bou99} allows us to compute $\nu_1^{-1}\pi_*X$, and the answer coincides with the answer obtained from computing $\pi_*TAQ(S_{K}^X)^\vee$ by means of the Kuhn filtration.
\end{proof}

\begin{Rem}
Note that one can easily check that all the differentials in the $TAQ$ spectral sequence from Theorem \ref{BousThm} commute with the map 
\begin{align*}
E_*(\ ) \stackrel{\psi^l-1}{\to} E_*(\ ).
\end{align*}
This allows us to get a complete description of the $\nu_1^{-1}\pi_*$ based Goodwillie spectral sequence for a $1$-connected $H$-space.

One could hope that this would allow us to give differentials in the unlocalized Goodwillie spectral sequence for the same space.

A different strategy for computing $E_*TAQ(S^X_{K(1)})$ would be computing $AQ^*_\mathbb{T}(E^*X)$, and then run the Basterra spectral sequence from \ref{BastSS}. In the case of $X$ being as in Theorem \ref{BousThm} this would have given the same answer, but not revealed anything about the Goodwillie spectral sequence.
\end{Rem}

\bibliography{bibleo}

\providecommand{\bysame}{\leavevmode\hbox to3em{\hrulefill}\thinspace}
\providecommand{\MR}{\relax\ifhmode\unskip\space\fi MR }
\providecommand{\MRhref}[2]{%
  \href{http://www.ams.org/mathscinet-getitem?mr=#1}{#2}
}
\providecommand{\href}[2]{#2}
\begin{thebibliography}{{Bra}17}

\bibitem[AM99]{AM99}
Greg Arone and Mark Mahowald, \emph{The {G}oodwillie tower of the identity
  functor and the unstable periodic homotopy of spheres}, Inventiones
  mathematicae \textbf{135} (1999), no.~3, 743--788.

\bibitem[Bas99]{Bas99}
Maria Basterra, \emph{{A}ndr{\'e}--{Q}uillen cohomology of commutative
  ${S}$-algebras}, Journal of Pure and Applied Algebra \textbf{144} (1999),
  no.~2, 111--143.

\bibitem[BH16]{BrHe16}
Lukas Brantner and Gijs Heuts, \emph{The $v_n$-periodic {G}oodwillie tower on
  wedges and cofibres}, arXiv preprint arXiv:1612.02694 (2016).

\bibitem[BM11]{BaMa11}
Maria Basterra and Michael~A Mandell, \emph{Homology of {E}n ring spectra and
  iterated {THH}}, Algebraic \& Geometric Topology \textbf{11} (2011), no.~2,
  939--981.

\bibitem[Bou99]{Bou99}
A{K} Bousfield, \emph{The {K}-theory localizations and $v_1$-periodic homotopy
  groups of {H}-spaces}, Topology \textbf{38} (1999), no.~6, 1239--1264.

\bibitem[BR17a]{BeRe17}
Mark Behrens and Charles Rezk, \emph{The {B}ousfield-{K}uhn functor and
  topological {A}ndre-{Q}uillen cohomology}, arXiv preprint arXiv:1712.03045
  (2017).

\bibitem[BR17b]{BeRe17Survey}
\bysame, \emph{Spectral algebra models of unstable $v_n$-periodic homotopy
  theory}, arXiv preprint arXiv:1703.02186 (2017).

\bibitem[{Bra}17]{Bra17}
Lukas {Brantner}, \emph{The {L}ubin-{T}ate theory of spectral lie algebras},
  Ph.D. thesis, 2017.

\bibitem[Chi05]{Chi05}
Michael Ching, \emph{Bar constructions for topological operads and the
  {G}oodwillie derivatives of the identity}, Geometry \& Topology \textbf{9}
  (2005), no.~2, 833--934.

\bibitem[Dav95]{Dav95}
Donald~M Davis, \emph{Computing $v_1$-periodic homotopy groups of spheres and
  some compact {L}ie groups}, Handbook of Algebraic Topology, North--Holland,
  Amsterdam (1995), 993--1048.

\bibitem[Fre00]{Fre00}
Benoit Fresse, \emph{On the homotopy of simplicial algebras over an operad},
  Transactions of the American Mathematical Society \textbf{352} (2000), no.~9,
  4113--4141.

\bibitem[Fre04]{Fre04}
\bysame, \emph{Koszul duality of operads}, Cont. Math \textbf{346} (2004),
  115--215.

\bibitem[Goo03]{Goo03}
Thomas~G Goodwillie, \emph{{C}alculus {III}: {T}aylor series}, Geometry \&
  Topology \textbf{7} (2003), 645--711.

\bibitem[Har10]{Har10}
John~E Harper, \emph{Bar constructions and {Q}uillen homology of modules over
  operads}, Algebraic \& Geometric Topology \textbf{10} (2010), no.~1, 87--136.

\bibitem[Heu18]{Heu18}
Gijs Heuts, \emph{Lie algebras and $ v\_n $-periodic spaces}, arXiv preprint
  arXiv:1803.06325 (2018).

\bibitem[HSS00]{HSS00}
Mark Hovey, Brooke Shipley, and Jeff Smith, \emph{Symmetric spectra}, Journal
  of the American Mathematical Society \textbf{13} (2000), no.~1, 149--208.

\bibitem[Kuh07]{Kuh07}
Nicholas~J Kuhn, \emph{Goodwillie towers and chromatic homotopy: an overview},
  Proceedings of the Nishida Fest (Kinosaki 2003) \textbf{10} (2007), 245--279.

\bibitem[Kuh08]{Kuh08}
\bysame, \emph{A guide to telescopic functors}, Homology, Homotopy and
  Applications \textbf{10} (2008), no.~3, 291--319.

\bibitem[Lin78]{Lin78}
James~{P} Lin, \emph{Torsion in {H}-spaces {II}}, Annals of Mathematics
  \textbf{107} (1978), no.~1, 41--88.

\bibitem[LV12]{LoVa12}
Jean-Louis Loday and Bruno Vallette, \emph{Algebraic operads}, vol. 346,
  Springer Science \& Business Media, 2012.

\bibitem[Mah67]{Mah67}
Mark~E Mahowald, \emph{The metastable homotopy of ${S}^n$}, no.~72, American
  Mathematical Soc., 1967.

\bibitem[Man01]{Man01}
Michael~A Mandell, \emph{$e_\infty$ algebras and $p$-adic homotopy theory},
  Topology \textbf{40} (2001), no.~1, 43--94.

\bibitem[McC83]{McC83}
James~E McClure, \emph{{D}yer-{L}ashof operations in ${K}$-theory}, Bulletin
  (New Series) of the American Mathematical Society \textbf{8} (1983), no.~1,
  67--72.

\bibitem[MSS02]{MSS02}
Martin Markl, Steven Shnider, and James~D Stasheff, \emph{Operads in algebra,
  topology and physics}, no.~96, American Mathematical Soc., 2002.

\bibitem[Qui69]{Qui69}
Daniel Quillen, \emph{Rational homotopy theory}, Annals of Mathematics (1969),
  205--295.

\bibitem[Rez09]{Rez09}
Charles Rezk, \emph{The congruence criterion for power operations in {M}orava
  {E}-theory}, Homology, Homotopy and Applications \textbf{11} (2009), no.~2,
  327--379.

\bibitem[Rez12]{Rez12}
\bysame, \emph{Rings of power operations for {M}orava {E}-theories are
  {K}oszul}, arXiv preprint arXiv:1204.4831 (2012).

\bibitem[Sul77]{Sul77}
Dennis Sullivan, \emph{Infinitesimal computations in topology}, Publications
  Math{\'e}matiques de l'IH{\'E}S \textbf{47} (1977), 269--331.

\end{thebibliography}
\bibliographystyle{amsalpha}
\end{document}